\documentclass[a4paper,twoside,fleqn]{article}
\usepackage{amsfonts}
\usepackage[dvips]{graphics}
\usepackage{amssymb}
\usepackage[reqno]{amsmath}
\usepackage{amsthm}
\usepackage{makeidx}
\usepackage{color}
\usepackage{graphicx}
\usepackage{subfigure}
\usepackage{multicol}
\usepackage{fancyhdr}
\usepackage[colorlinks,linkcolor=blue,anchorcolor=blue,citecolor=blue]{hyperref}

 \newtheorem{theorem}{\sc\bf Theorem}[section]

 \newtheorem{lemma}[theorem]{\sc\bf Lemma}
 
 \newtheorem{definition}[theorem]{\sc\bf Definition}

  \numberwithin{equation}{section}

\makeatletter
\def\@cite#1#2{#1\if@tempswa , #2\fi}
\makeatother

\title{{\bf Common properties of bounded linear operators \\ $AC$ and $BA$: Spectral theory}
\thanks{This work has been supported by National Natural Science Foundation
of China (11171066, 11201071, 11226113), Specialized Research Fund for the Doctoral
Program of Higher Education (20103503110001, 20113503120003), Natural
Science Foundation of Fujian Province (2011J05002, 2012J05003) and
Foundation of the Education Department of Fujian Province
(JA12074).}}

\author{Qingping \textsc{Zeng}\thanks{Email address: zqpping2003@163.com.}
 \quad Huaijie \textsc{Zhong}
}

\begin{document}
\date{}
\maketitle

\large

\begin{quote}
 {\bf Abstract:}~Let $X,Y$ be Banach spaces, $A:X \longrightarrow Y$ and $B,C:Y \longrightarrow X$ be
 bounded linear operators satisfying operator equation $ABA=ACA$. Recently, as extensions of Jacobson's lemma, Corach, Duggal and Harte studied common properties of $AC-I$ and $BA-I$ in algebraic viewpoint and also obtained some topological analogues. In this note, we continue to investigate common properties of $AC$ and $BA$ from the viewpoint of spectral theory. In particular, we give an affirmative answer to one question posed by Corach et al. by proving that $AC - I$ has closed range if and only if $BA - I$ has closed range.   
\end{quote}

\section{Introduction and Notations}

For any Banach spaces $X$ and $Y$, let $\mathcal{B}(X,Y)$ denote the set of all bounded linear operators from $X$ to $Y$.
Jacobson's lemma [\cite{Aiena-Gonzalez,Barnes,Benhida-Zerouali,Harte-book,Lin common,Muller-book}] states that if $A \in \mathcal{B}(X,Y)$ and $C \in \mathcal{B}(Y,X)$ then
\begin{equation}\label{eq1.1} \qquad\qquad\qquad
 AC-I \makebox{ is invertible} \Longleftrightarrow  CA-I \makebox{ is invertible}.
\end{equation}
Recently, we generalized (\ref{eq1.1}) to various regularities, in the sense of Kordula and M\"uller, and showed that $AC-I$ and $CA-I$ share common complementability of kernels (see [\cite{Zeng-Zhong-RS-SR}]).
In 2013, Corach, Duggal and Harte [\cite{Corach-Duggal-Harte}] extended (\ref{eq1.1}) and many of its relatives form $CA-I$ to certain $BA-I$ under the assumption
\begin{equation}\label{eq1.2} \qquad\qquad\qquad\qquad\qquad\qquad\qquad
ABA=ACA,
\end{equation}
where $A \in \mathcal{B}(X,Y)$ and $B, C \in \mathcal{B}(Y,X)$. In this note, we continue to study this situation and show that $AC$ and $BA$ share many common spectral properties. For example, answering one question posed by Corach et al. in [\cite{Corach-Duggal-Harte}, p. 526], we prove that $AC - I$ has closed range if and only if $BA - I$ has closed range. But at present we are unable to decide whether $AC-I$ and $BA-I$ share common complementability of kernels. For some other open questions in this direction, we refer the reader to [\cite{Corach-Duggal-Harte,Zeng-Zhong-RS-SR}].

We first fix some natations in spectral theory. Throughout this paper, $\mathcal{B}(X) = \mathcal{B}(X,X)$. For an
operator $T \in \mathcal{B}(X)$, let $\mathcal {N}(T)$ denote its kernel, $\mathcal{R}(T)$ its range and $\sigma(T)$ its spectrum. For each $n \in \mathbb{N}:=\{0,1,2,\cdots\}$, we set $c_{n}(T) = \dim \mathcal
{R}(T^{n})/\mathcal {R}(T^{n+1})$ and $c^{'}_{n}(T) = \dim \mathcal
{N}(T^{n+1})/\mathcal {N}(T^{n}).$ It is well known that ([\cite{Kaashoek-M-A-12}, Lemmas 3.2 and 3.1]), for every
$n \in \mathbb{N}$,
$$c_{n}(T) = \dim X / (\mathcal {R}(T) + \mathcal {N}(T^{n})), \ \ \ \  c^{'}_{n}(T) =  \dim \mathcal {N}(T)
\cap \mathcal {R}(T^{n}).$$
Hence, it is easy to see that the
sequences $\{c_{n}(T)\}_{n=0}^{\infty}$ and
$\{c^{'}_{n}(T)\}_{n=0}^{\infty}$ are decreasing.
For each $n \in \mathbb{N}$, $T$ induces a
linear transformation from the vector space $\mathcal {R}(T^{n})/
\mathcal {R}(T^{n+1})$ to the space $\mathcal {R}(T^{n+1})/\mathcal
{R}(T^{n+2})$. We will let $k_{n}(T)$ be the dimension of the null
space of the induced map and let
$$k(T) = \sum_{n=0}^{\infty} k_{n}(T).$$ From Lemma 2.3 in [\cite{Grabiner 9}]
it follows that, for every $n \in \mathbb{N}$,
\begin{align*} \qquad  \qquad \qquad  k_{n}(T)
 &= \dim (\mathcal {N}(T) \cap \mathcal {R}(T^{n})) / (\mathcal {N}(T) \cap \mathcal {R}(T^{n+1})) \\
 &= \dim (\mathcal {R}(T) +
\mathcal {N}(T^{n+1})) / (\mathcal {R}(T) + \mathcal {N}(T^{n})).
 \end{align*}
We remark that the sequence $\{k_{n}(T)\}_{n=0}^{\infty}$ is not
always decreasing. From Theorem 3.7 in [\cite{Grabiner 9}]
it follows that
$$k(T) = \dim \mathcal {N}(T)/(\mathcal {N}(T) \cap \mathcal
{R}(T^{\infty})) = \dim (\mathcal {R}(T) + \mathcal
{N}(T^{\infty}))/\mathcal {R}(T).$$

Just as the definition of $k(T)$, we give the definitions of
$stable$ $nullity$ $c^{'}(T)$ and $stable$ $defect$ $c(T)$ as
follows.

\begin {definition}
 Let $T \in \mathcal{B}(X)$. The $stable$ $nullity$ $c^{'}(T)$ of $T$ is defined as
 $$c^{'}(T)= \sum_{n=0}^{\infty} c_{n}^{'}(T),$$
 and the $stable$ $defect$ $c(T)$ of $T$ is defined as $$c(T)= \sum_{n=0}^{\infty}
 c_{n}(T).$$
\end{definition}

It is easy to see that $c(T)= \dim X/\mathcal {R}(T^{\infty}) $ and
$c^{'}(T)= \dim \mathcal {N}(T^{\infty})$.

In [\cite{Kordula-Muller}],  Kordula and M\"uller defined the
concept of $regularity$ as follows:

\begin {definition} \begin {upshape}
([\cite{Kordula-Muller}]) \end {upshape}
 A non-empty subset ${\bf R} \subseteq \mathcal{B}(X)$ is called a
 $regularity$ if it satisfies the following conditions:

 $(1)$ If $A \in \mathcal{B}(X)$ and $n \geq 1$, then $A \in {\bf
 R}$ if and only if $A^{n} \in {\bf R}$.

 $(2)$ If $A, B, C, D \in \mathcal{B}(X)$ are mutually commuting operators satisfying $AC+BD=I$, then $AB \in {\bf
 R}$ if and only if $A,B \in {\bf R}$.
\end{definition}

A non-empty subset ${\bf R} \subseteq \mathcal{B}(X)$ defines in a
natural way a spectrum by $\sigma_{{\bf R}}(T) = \{ \lambda \in
\mathbb{C}: \lambda I - T \notin {\bf R} \} $, for every $T \in
\mathcal{B}(X)$. The crucial property of the spectrum $\sigma_{{\bf R}}$ corresponding to a regularity ${\bf R}$ is that it satisfies a restricted spectral mapping theorem ([\cite{Kordula-Muller}, Theorem 1.4]),
$$f(\sigma_{{\bf R}}(T))=\sigma_{{\bf R}}(f(T))$$
for every function $f$ analytic on a neighbourhood of $\sigma(T)$ which is non-constant on each component of its domain of definition.

We now give the definitions of some concrete subsets ${\bf R_{i}} \subseteq \mathcal{B}(X)$, $1
\leq i \leq 19$. For the fact that ${\bf R_{i}}$ ($1\leq i \leq 19$) form a regularity, the reader should refer to [\cite{Berkani,Mbekhta-Muller 9,Zeng-Zhong-RS-SR}].

\begin {definition}  ${\bf R_{1}} = \{T \in \mathcal{B}(X): c(T) = 0 \}$,

${\bf R_{2}} = \{T \in \mathcal{B}(X): c(T) < \infty \}$,

${\bf R_{3}} = \{T \in \mathcal{B}(X): \makebox{there exists } d \in
\mathbb{N} \makebox{ such that } c_{d}(T) = 0 \makebox{ and}$
$\mathcal {R}(T^{d+1}) \makebox{ is closed} \}$,

${\bf R_{4}} =  \{T \in \mathcal{B}(X): c_{n}(T) < \infty \makebox{
for every } n \in \mathbb{N} \} $,

${\bf R_{5}} = \{T \in \mathcal{B}(X): \makebox{there exists } d \in
\mathbb{N} \makebox{ such that } c_{d}(T) < \infty \makebox{ and}$
$\mathcal {R}(T^{d+1}) \makebox{ is closed} \}$,

${\bf R_{6}} = \{T \in \mathcal{B}(X): c^{'}(T) = 0 \makebox{ and }
\mathcal {R}(T) \makebox{ is closed} \}$,

${\bf R_{7}} = \{T \in \mathcal{B}(X): c^{'}(T) < \infty \makebox{
and } \mathcal {R}(T) \makebox{ is closed}  \}$,

${\bf R_{8}} = \{T \in \mathcal{B}(X): \makebox{there exists } d \in
\mathbb{N} \makebox{ such that } c_{d}^{'}(T) = 0 \makebox{ and}$
$\mathcal {R}(T^{d+1}) \makebox{ is closed} \}$,

${\bf R_{9}} = \{T \in \mathcal{B}(X): c_{n}^{'}(T) < \infty
\makebox{ for every } n \in \mathbb{N} \makebox{ and } \mathcal
{R}(T) \makebox{ is closed} \} $,

${\bf R_{10}} = \{T \in \mathcal{B}(X): \makebox{there exists } d
\in \mathbb{N} \makebox{ such that } c_{d}^{'}(T) < \infty \makebox{
and}$ $\mathcal {R}(T^{d+1}) \makebox{ is closed} \}$,

${\bf R_{11}} = \{T \in \mathcal{B}(X): k(T) = 0  \makebox{ and }
\mathcal {R}(T) \makebox{ is closed} \}$,

${\bf R_{12}} = \{T \in \mathcal{B}(X): k(T) < \infty  \makebox{ and
} \mathcal {R}(T) \makebox{ is closed} \}$,

${\bf R_{13}} =\{T \in \mathcal{B}(X): \makebox{there exists } d \in
\mathbb{N} \makebox{ such that } k_{n}(T) = 0 \makebox{ for every }
n \geq d  \makebox{ and}$ $\mathcal {R}(T^{d+1}) \makebox{ is}$
$\makebox{closed} \}$,

${\bf R_{14}} = \{T \in \mathcal{B}(X): k_{n}(T) < \infty \makebox{
for every } n \in \mathbb{N} \makebox{ and } \mathcal {R}(T)$
$\makebox{is} \makebox{closed} \}$,

${\bf R_{15}} = \{T \in \mathcal{B}(X):  \makebox{there exists } d
\in \mathbb{N} \makebox{ such that } k_{n}(T) < \infty \makebox{ for
every } n \geq d \makebox{ and } $ $ \mathcal {R}(T^{d+1})$ $
\makebox{ is closed} \}$.

${\bf R_{16}} = \{T \in \mathcal{B}(X):  \makebox{there exists } d
\in \mathbb{N} \makebox{ such that } c_{d}(T) = 0 \makebox{ and}$
$\mathcal {R}(T) + \mathcal {N}(T^{d}) \makebox{ is} $ $
\makebox{closed} \}$,

${\bf R_{17}} = \{T \in \mathcal{B}(X): \makebox{there exists } d
\in \mathbb{N} \makebox{ such that } c_{d}(T) < \infty \makebox{
and}$ $\mathcal {R}(T) + \mathcal {N}(T^{d}) \makebox{ is} $ $
\makebox{closed} \}$,

${\bf R_{18}} = \{T \in \mathcal{B}(X):
\makebox{there exists } d \in \mathbb{N} \makebox{ such that }
k_{n}(T) = 0 \makebox{ for every } n \geq d, \makebox{ and}$
$\mathcal {R}(T) + \mathcal {N}(T^{d}) \makebox{ is closed} \}$,

${\bf R_{19}} = \{T \in \mathcal{B}(X): \makebox{there exists } d
\in \mathbb{N} \makebox{ such that } k_{n}(T) < \infty \makebox{ for
every } n \geq d, \makebox{ and } $ $\mathcal {R}(T) + \mathcal
{N}(T^{d}) \makebox{ is}$ $ \makebox{closed} \}$.
\end{definition}

We remark that $\bf R_{1} \subseteq \bf R_{2}  = \bf R_{3} \cap \bf R_{4}
\subseteq \bf R_{3} \cup \bf R_{4}  \subseteq \bf R_{5} \subseteq
\bf R_{13} \subseteq \bf R_{18}$, $\bf R_{3} \subseteq \bf R_{16}$, $\bf R_{5} \subseteq \bf R_{17}$, $\bf R_{6} \subseteq \bf R_{7} = \bf R_{8} \cap \bf R_{9} \subseteq \bf R_{8} \cup \bf R_{9}  \subseteq \bf R_{10}
\subseteq \bf R_{13}$, $\bf R_{11} \subseteq \bf R_{12} = \bf
R_{13} \cap \bf R_{14} \subseteq \bf R_{13} \cup \bf R_{14}
\subseteq \bf R_{15} \subseteq \bf R_{19}$. The operators of $\bf R_{1}$, $\bf R_{2}$, $\bf R_{3}$, $\bf R_{4}$ and $\bf R_{5}$ are called surjective, lower semi-Browder, right Drazin invertible, lower semi-Fredholm and right essentially Drazin invertible operators, respectively.
The operators of $\bf R_{6}$, $\bf R_{7}$, $\bf R_{8}$, $\bf R_{9}$ and $\bf R_{10}$ are called bounded below, upper semi-Browder, left Drazin invertible, upper semi-Fredholm and left essentially Drazin invertible operators, respectively. The operators of $\bf R_{11}$, $\bf R_{12}$ and $\bf R_{13}$ are called semi-regular, essentially semi-regular and quasi-Fredholm operators, respectively. The operators of $\bf R_{18}$ are called operators with eventual topological uniform descent.

The main result of this note establishes that if $A \in \mathcal{B}(X,Y)$ and $B, C \in \mathcal{B}(Y,X)$ satisfy (\ref{eq1.2}), then
 $$AC-I \in {\bf R_{i}} \Longleftrightarrow BA-I \in {\bf R_{i}}, \ \ 1 \leq i \leq 19.$$
It not only extends our previous corresponding results in [\cite{Zeng-Zhong-RS-SR}] from the special case
$$B=C$$
to the general case,
but also supplements the results obtained by Corach, Duggal and Harte in [\cite{Corach-Duggal-Harte}] from the viewpoint of spectral theory.

\section{Main result}

Throughout this section, we assume that $A \in \mathcal{B}(X,Y)$ and $B, C \in \mathcal{B}(Y,X)$ satisfy (\ref{eq1.2}). We begin with the following lemma, which gives an affirmative answer to one question posed by Corach et al. in [\cite{Corach-Duggal-Harte}, p. 526].

\begin {lemma} \label{2.1}

$\mathcal {R}(AC-I)$ is closed if and only if $\mathcal {R}(BA-I)$ is closed.

\end{lemma}

\begin{proof} We give the proof by taking the argument in [\cite{Corach-Duggal-Harte}, pp. 526-527] one step further. Suppose that $\mathcal {R}(AC-I)$ is closed. Assume that there exists a sequence
$\{x_{n}\}_{n=1}^{\infty} \subseteq \mathcal {R}(BA-I)$ such that $x_{n} \rightarrow x$. Then, for each
positive integer $n$, there exists $z_{n} \in X$ such that $x_{n}=(BA-I)z_{n}$. Hence
$$Ax=\lim\limits_{n \rightarrow \infty}Ax_{n}=\lim\limits_{n \rightarrow \infty}A(BA-I)z_{n} =\lim\limits_{n \rightarrow \infty}(AC-I)A z_{n}.$$
Since $\mathcal {R}(AC-I)$ is closed, there exists $y \in Y$ such that
$$Ax = (AC-I)y,$$
and hence $y=ACy-Ax.$ Therefore,
\begin{align*} \qquad\qquad
      x  &= BAx-(BA-I)x \\
         &= B(AC-I)y-(BA-I)x  \\
         &= (BAC-B)(ACy-Ax)-(BA-I)x \\
         &= BACACy-BACAx-BACy+BAx-(BA-I)x \\
         &= BABACy-BABAx-BACy+BAx-(BA-I)x \\
         &= (BA-I)(BACy-BAx-x),
     \end{align*}
that is, $x \in \mathcal {R}(BA-I)$. Consequently, $\mathcal {R}(BA-I)$ is closed.

Conversely, suppose that $\mathcal {R}(BA-I)$ is closed. Then we have
    \begin{align*}
  \mathcal {R}(BA-I) \makebox{ is closed} &\Longleftrightarrow \mathcal {R}(AB-I) \makebox{ is closed} \ \ (\mathrm{by \ [\cite{Barnes}, \  Theorem \ 5]})\\
         &\Longrightarrow \mathcal {R}(CA-I) \makebox{ is closed}  \\
          &\qquad\,(\mathrm{by \ the \ previous \ paragraph \ and \ interchanging \ } B \mathrm{ \ and \ } C) \\
         &\Longleftrightarrow \mathcal {R}(AC-I) \makebox{ is closed}. \ \ (\mathrm{by \ [\cite{Barnes}, \ Theorem \ 5]})
     \end{align*}
Hence, $\mathcal {R}(AC-I)$ is closed $\Longleftrightarrow$ $\mathcal {R}(BA-I)$ is closed.
\end{proof}

\begin {lemma} \label{2.2}

For all $n \in \mathbb{N}$, $\mathcal {R}((AC-I)^{n})$ is closed if and only if $\mathcal {R}((BA-I)^{n})$ is closed.

\end{lemma}

\begin{proof} For each $n \in \mathbb{N}$, let
  $$B_{n} = \sum_{k=1}^{n+1}(-1)^{k-1} \binom{n+1}{k} B(AB)^{k-1} $$
  and
  $$C_{n} = \sum_{k=1}^{n+1}(-1)^{k-1} \binom{n+1}{k} C(AC)^{k-1}.$$
  Then we have
   $$AB_{n}A =AC_{n}A,$$
   $$(I-AC)^{n+1} =I-AC_{n}$$
   and
   $$(I-BA)^{n+1} =I-B_{n}A.$$
   Applying Lemma \ref{2.1}, we get
   $$\mathcal {R}((AC-I)^{n}) \makebox{ is closed} \Longleftrightarrow \mathcal {R}((BA-I)^{n}) \makebox{ is closed},$$
   for all $n \in \mathbb{N}$.
\end{proof}

The next lemma is essential to the sequel crucial lemmas \ref{2.4}--\ref{2.7}.

\begin {lemma} \label{2.3} For all $d \in \mathbb{N}$, we have

$(1)$ $A\mathcal {R}((BA-I)^{d}) \subseteq \mathcal {R}((AC-I)^{d});$

$(2)$ $A\mathcal {N}((BA-I)^{d}) \subseteq \mathcal {N}((AC-I)^{d});$

$(3)$ $BAC\mathcal {N}((AC-I)^{d}) \subseteq \mathcal {N}((BA-I)^{d});$

$(4)$ $BAC\mathcal {R}((AC-I)^{d}) \subseteq \mathcal {R}((BA-I)^{d}).$

\end{lemma}

\begin{proof}  (1) Let $x \in \mathcal {R}((BA-I)^{d})$. Then there exists $x_{0} \in X$ such that $x=(BA-I)^{d}x_{0}$, hence
 $$Ax=A(BA-I)^{d}x_{0}=(AC-I)^{d}Ax_{0} \in \mathcal {R}((AC-I)^{d}).$$
 Therefore, $A\mathcal {R}((BA-I)^{d}) \subseteq \mathcal {R}((AC-I)^{d}).$

 (2) Let $x \in \mathcal {N}((BA-I)^{d})$. Then we have
    $$(AC-I)^{d}Ax=A(BA-I)^{d}x=0.$$
    Thus $Ax \in \mathcal {N}((AC-I)^{d})$ and this shows (2).

 (3) Let $y \in \mathcal {N}((AC-I)^{d})$. Then we have
    $$(BA-I)^{d}BACy=BAC(AC-I)^{d}y=0.$$
    Thus $BACy \in \mathcal {N}((BA-I)^{d})$ and this shows (3).

 (4) Let $y \in \mathcal {R}((AC-I)^{d})$. Then there exists $y_{0} \in Y$ such that $y=(AC-I)^{d}y_{0}$, hence
    $$BACy=BAC(AC-I)^{d}y_{0}=(BA-I)^{d}BACy_{0} \in \mathcal {R}((BA-I)^{d}).$$
     Therefore, $BAC\mathcal {R}((AC-I)^{d}) \subseteq \mathcal {R}((BA-I)^{d}).$
\end{proof}

The proofs of the following lemmas \ref{2.4}--\ref{2.7} are dependent heavily on the special case
$$B=C,$$
which we proved recently in [\cite{Zeng-Zhong-RS-SR}].

\begin {lemma} \label{2.4}

For all $d \in \mathbb{N}$, $\mathcal {R}(AC-I) + \mathcal {N}((AC-I)^{d})$ is closed if and only if $\mathcal {R}(BA-I) + \mathcal {N}((BA-I)^{d})$ is closed.

\end{lemma}

\begin{proof} Suppose that $\mathcal {R}(AC-I) + \mathcal {N}((AC-I)^{d})$ is closed. Assume that there exists a sequence
$\{x_{n}\}_{n=1}^{\infty} \subseteq \mathcal {R}(BA-I) + \mathcal {N}((BA-I)^{d})$ such that $x_{n} \rightarrow x$. Then, for each
positive integer $n$, there exist $y_{n} \in \mathcal {R}(BA-I)$ and $z_{n} \in \mathcal {N}((BA-I)^{d})$ such that $x_{n}=y_{n}+z_{n}$. Hence
$$Ax=\lim\limits_{n \rightarrow \infty}Ax_{n}=\lim\limits_{n \rightarrow \infty}A(y_{n}+z_{n}).$$
Parts (1) and (2) of Lemma \ref{2.3} imply that
$$Ay_{n} \in \mathcal {R}(AC-I)$$
and
$$Az_{n} \in \mathcal {N}((AC-I)^{d}).$$
Since $\mathcal {R}(AC-I) + \mathcal {N}((AC-I)^{d})$ is closed, there exist $y \in Y$ and $z \in \mathcal {N}((AC-I)^{d})$ such that
$$Ax = (AC-I)y+z,$$
and hence $y=ACy+z-Ax.$ Therefore,
\begin{align*}
      x  &= BAx-(BA-I)x \\
         &= B[(AC-I)y+z]-(BA-I)x  \\
         &= (BAC-B)(ACy+z-Ax)+Bz-(BA-I)x \\
         &= BACACy+BACz-BACAx-BACy-Bz+BAx+Bz-(BA-I)x \\
         &= BABACy+BACz-BABAx-BACy+BAx-(BA-I)x \\
         &= (BA-I)(BACy-BAx-x)+BACz.
     \end{align*}
And then, since $z \in \mathcal {N}((AC-I)^{d})$, we get $ x  \in \mathcal {R}(BA-I) + \mathcal {N}((BA-I)^{d})$ by Lemma \ref{2.3}(3).
Consequently, $\mathcal {R}(BA-I) + \mathcal {N}((BA-I)^{d})$ is closed.

Conversely, suppose that $\mathcal {R}(BA-I) + \mathcal {N}((BA-I)^{d})$ is closed. Then we have
    \begin{align*}\qquad\,
  &\qquad\, \mathcal {R}(BA-I) + \mathcal {N}((BA-I)^{d}) \makebox{ is closed} \\
  & \Longleftrightarrow \mathcal {R}(AB-I) + \mathcal {N}((AB-I)^{d}) \makebox{ is closed} \ \ (\mathrm{by \ [\cite{Zeng-Zhong-RS-SR}, \  Lemma \ 3.11]})\\
  &\Longrightarrow \mathcal {R}(CA-I)+ \mathcal {N}((CA-I)^{d}) \makebox{ is closed}  \\
  &\qquad\,(\mathrm{by \ the \ previous \ paragraph \ and \ interchanging \ } B \mathrm{ \ and \ } C) \\
  &\Longleftrightarrow \mathcal {R}(AC-I)+ \mathcal {N}((AC-I)^{d}) \makebox{ is closed}. \ \ (\mathrm{by \ [\cite{Zeng-Zhong-RS-SR}, \ Lemma \ 3.11]})
     \end{align*}
Hence, $\mathcal {R}(AC-I) + \mathcal {N}((AC-I)^{d})$ is closed $\Longleftrightarrow$ $\mathcal {R}(BA-I) + \mathcal {N}((BA-I)^{d})$ is closed.
\end{proof}

\begin {lemma}${\label{2.5}}$ $c_{n}^{'}(AC-I)=c_{n}^{'}(BA-I)$ for all $n \in \mathbb{N}$. Consequently, $c^{'}(AC-I)=c^{'}(BA-I)$.
\end{lemma}

\begin{proof}  Let $\widehat{A}(c_{n}^{'})$ be the linear mapping
induced by $A$ from $$\mathcal {N}((BA-I)^{n+1})/ \mathcal
{N}((BA-I)^{n})$$ to $$\mathcal {N}((AC-I)^{n+1}) / \mathcal
{N}((AC-I)^{n}).$$
Since $A\mathcal {N}((BA-I)^{n}) \subseteq \mathcal{N}((AC-I)^{n})$ (see Lemma \ref{2.3}(2)),
we thus know that $\widehat{A}(c_{n}^{'})$ is well defined.

Next, we show that $\widehat{A}(c_{n}^{'})$ is injective. In fact, let $x \in
\mathcal {N}((BA-I)^{n+1})$ and $Ax \in \mathcal{N}((AC-I)^{n})$. Then by Lemma \ref{2.3}(3), we have $BACAx \in \mathcal{N}((BA-I)^{n}).$
Hence, \begin{align*}\qquad\,\qquad\,\qquad\,
  x&= BAx -(BA-I)x \\
   &=BAx-BACAx+BACAx -(BA-I)x \\
   & =BAx-BABAx+BACAx -(BA-I)x \\
   &=BA(I-BA)x+BACAx-(BA-I)x   \\
   &\in \mathcal{N}((BA-I)^{n}).
     \end{align*}
Therefore,
$\widehat{A}(c_{n}^{'})$ is injective.

Hence,  \begin{align*} \qquad\,
  c_{n}^{'}(BA-I) &\leq c_{n}^{'}(AC-I) \ \ (\mathrm{by \ the \ previous \ paragraph})\\
         &=c_{n}^{'}(CA-I) \ \ (\mathrm{by \ [\cite{Zeng-Zhong-RS-SR}, \ Lemma \ 3.10]}) \\
         &\leq c_{n}^{'}(AB-I)  \\
         &\quad\,(\mathrm{by \ the \ previous \ paragraph \ and \ interchanging \ } B \mathrm{ \ and \ } C) \\
         &=c_{n}^{'}(BA-I) \ \ (\mathrm{by \ [\cite{Zeng-Zhong-RS-SR}, \ Lemma \ 3.10]})
     \end{align*}
So $c_{n}^{'}(AC-I) = c_{n}^{'}(BA-I)$ for all $n \in \mathbb{N}$.
\end{proof}

\begin {lemma}${\label{2.6}}$ $c_{n}(AC-I)=c_{n}(BA-I)$ for all $n \in
 \mathbb{N}$. Consequently, $c(AC-I)=c(BA-I)$.

\end{lemma}

\begin{proof} Let $\widehat{A}(c_{n})$ be the linear mapping induced
by $A$ from
$$X/ (\mathcal {R}(BA-I)+\mathcal{N}((BA-I)^{n})) $$
to
$$Y/ (\mathcal {R}(AC-I)+\mathcal{N}((AC-I)^{n})).$$
Since
$$A(\mathcal {R}(BA-I)+\mathcal{N}((BA-I)^{n})) \subseteq \mathcal {R}(AC-I)+\mathcal{N}((AC-I)^{n})$$
(by Lemma \ref{2.3}(1) and (2)), we thus know that $\widehat{A}(c_{n}^{'})$ is well defined.

Next, we show that $\widehat{A}(c_{n})$ is injective. In fact, let $x \in X$ and $Ax \in \mathcal {R}(AC-I)+\mathcal{N}((AC-I)^{n})$.
Then by the proof of the first paragraph in Lemma \ref{2.4}, we get $x \in \mathcal {R}(BA-I)+\mathcal{N}((BA-I)^{n}).$ Therefore, $\widehat{A}(c_{n}^{'})$ is injective.

Hence, \begin{align*} \qquad\,
  c_{n}(BA-I) &\leq c_{n}(AC-I) \ \ (\mathrm{by \ the \ previous \ paragraph})\\
         &=c_{n}(CA-I) \ \ (\mathrm{by \ [\cite{Zeng-Zhong-RS-SR}, \ Lemma \ 3.9]}) \\
         &\leq c_{n}(AB-I)  \\
         &\quad\,(\mathrm{by \ the \ previous \ paragraph \ and \ interchanging \ } B \mathrm{ \ and \ } C) \\
         &= c_{n}(BA-I) \ \ (\mathrm{by \ [\cite{Zeng-Zhong-RS-SR}, \ Lemma \ 3.9]})
     \end{align*}
So $c_{n}(AC-I) = c_{n}(BA-I)$ for all $n \in \mathbb{N}$.
\end{proof}

\begin {lemma}${\label{2.7}}$ $k_{n}(AC-I)=k_{n}(BA-I)$ for all $n \in
 \mathbb{N}$. Consequently, $k(AC-I)=k(BA-I)$.

\end{lemma}

\begin{proof} Let $\widehat{A}(k_{n})$ be the linear mapping
induced by $A$ from
$$(\mathcal {N}(BA-I) \cap \mathcal {R}((BA-I)^{n})) / (\mathcal {N}(BA-I) \cap \mathcal{R}((BA-I)^{n+1}))$$
to
$$(\mathcal {N}(AC-I) \cap \mathcal{R}((AC-I)^{n})) / (\mathcal {N}(AC-I) \cap \mathcal{R}((AC-I)^{n+1})).$$
Since
$$A(\mathcal {N}(BA-I) \cap \mathcal{R}((BA-I)^{n+1})) \subseteq \mathcal {N}(AC-I) \cap \mathcal{R}((AC-I)^{n+1})$$
(by Lemma \ref{2.3}(1) and (2)), we thus know that $\widehat{A}(k_{n})$ is well defined.

Next, we show that $\widehat{A}(k_{n})$ is injective. In fact, let $x \in \mathcal {N}(BA-I) \cap \mathcal {R}((BA-I)^{n})$ and
$Ax \in \mathcal {N}(AC-I) \cap \mathcal{R}((AC-I)^{n+1})$. Then by Lemma \ref{2.3}(4), we have $BACAx \in \mathcal{R}((BA-I)^{n+1}).$
Hence, \begin{align*}\qquad\,\qquad\,\qquad\,
  x&= BAx -(BA-I)x \\
   &=BAx-BACAx+BACAx -(BA-I)x \\
   & =BAx-BABAx+BACAx -(BA-I)x \\
   &=BA(I-BA)x+BACAx-(BA-I)x   \\
   &\in \mathcal{R}((BA-I)^{n+1}),
     \end{align*}
thus $x \in \mathcal {N}(BA-I) \cap \mathcal{R}((BA-I)^{n+1}).$
Therefore, $\widehat{A}(c_{n}^{'})$ is injective.

Hence, \begin{align*} \qquad\,
  k_{n}(BA-I) &\leq k_{n}(AC-I) \ \ (\mathrm{by \ the \ previous \ paragraph})\\
         &=k_{n}(CA-I) \ \ (\mathrm{by \ [\cite{Zeng-Zhong-RS-SR}, \ Lemma \ 3.8]}) \\
         &\leq k_{n}(AB-I)  \\
         &\quad\,(\mathrm{by \ the \ previous \ paragraph \ and \ interchanging \ } B \mathrm{ \ and \ } C) \\
         &= k_{n}(BA-I) \ \ (\mathrm{by \ [\cite{Zeng-Zhong-RS-SR}, \ Lemma \ 3.8]})
     \end{align*}
So $k_{n}(AC-I) = k_{n}(BA-I)$ for all $n \in \mathbb{N}$.
\end{proof}

Since the basic components of regularities ${\bf R_{i}}$ ($1\leq i \leq 19$) are all considered in
Lemmas \ref{2.1}, \ref{2.2} and \ref{2.4}--\ref{2.7},
we are now in a position to give the proof of the following main result.

\begin{theorem} ${\label{2.8}}$  $\sigma_{{\bf R_{i}}} (AC) \backslash \{0\} = \sigma_{{\bf
R_{i}}} (BA) \backslash \{0\}$ for $1 \leq i \leq 19$.
\end{theorem}

\begin{proof} Applying Lemmas \ref{2.1}, \ref{2.2} and \ref{2.4}--\ref{2.7}, the desired conclusion follows directly.
\end{proof}


\end{document}